\documentclass[10pt, english]{article}
\usepackage[utf8]{inputenc}
\usepackage[T1]{fontenc}

\usepackage[authoryear, round]{natbib}

\usepackage{amsmath}
\usepackage{amsfonts}
\usepackage{amssymb}
\usepackage{graphicx}
\usepackage{float}
\usepackage{relsize}
\usepackage{tikz}
\usetikzlibrary{arrows}

\usepackage{color}%
\providecommand{\U}[1]{\protect\rule{.1in}{.1in}}
\newtheorem{theorem}{Theorem}

\newtheorem{corollary}[theorem]{Corollary}

\newtheorem{example}[theorem]{Example}

\newenvironment{proof}[1][Proof]{\noindent\textbf{#1.} }{\ \rule{0.5em}{0.5em}}

\title{Condensation of the digraph associated with a reciprocal matrix and a vector} 
\author{Ros\'{a}rio Fernandes \thanks{Email: mrff@fct.unl.pt}\\Center for Mathematics and Applications (NOVA Math)  \\and Department of Mathematics \\NOVA School of Science and Technology (NOVA FCT)
\\2829-516 Caparica, Portugal}
\date{\today}

\usepackage[utf8]{inputenc}
\usepackage[english]{babel}

\usepackage{natbib}
\setcitestyle{authoryear,open={(},close={)}} 

\begin{document}

\maketitle

\begin{abstract}
Reciprocal matrices are a fundamental tool in the Analytic Hierarchy Process (AHP), where priority vectors are typically derived from pairwise comparisons. The efficiency of a positive vector, in the sense of Pareto optimality, can be characterized through the strong connectivity of a directed graph $G_{A,w}$ associated with a reciprocal matrix $A$ and a vector $w$.

In this paper, we investigate the structure of the condensation digraph of $G_{A,w}$ in the case where $w$ is inefficient, with particular emphasis on the Perron vector. We provide a characterization of this structure and derive several constructive results. In particular, we show how efficiency can be achieved by modifying a single pair of reciprocal entries, and we construct an augmented reciprocal matrix whose efficient vector extends $w$. Finally, we propose a procedure to transform $w$ into an efficient vector for $A$.
\end{abstract}



\section{Introduction}

\hspace{3ex}Pairwise comparisons constitute one of the oldest formal approaches to preference modelling and multicriteria decision analysis. Their origins can be traced back to the thirteenth century, and their intuitive nature has made them one of the most successful methodologies for eliciting preferences among competing alternatives. Over the last five decades, pairwise comparison methods have evolved into a mature mathematical discipline at the intersection of linear algebra, optimization, graph theory and decision sciences.

Among methodologies based on pairwise comparisons, the Analytic Hierarchy Process (AHP), originally proposed by \citet{saaty1977scaling} and further developed in \citep{saaty1980,saaty1984comparison}, has become the dominant framework for multicriteria decision making. Owing to its simplicity and flexibility, AHP has found applications in engineering, management, economics, environmental sciences, logistics, healthcare and many other areas; see, for example, \cite{thakkar2021multi}. Beyond its practical success, AHP has stimulated extensive mathematical research on the properties of reciprocal matrices and the derivation of priority vectors.

The central mathematical object in AHP is the reciprocal pairwise comparison matrix, whose entries quantify the relative preference between pairs of alternatives or criteria. When the judgements are perfectly consistent, the associated priority vector is uniquely determined. In practice, however, inconsistency is almost unavoidable, giving rise to one of the fundamental questions in AHP: how should priorities be derived from an inconsistent reciprocal matrix?

Saaty answered this question by proposing the normalized positive eigenvector associated with the largest eigenvalue as the priority vector associated with a reciprocal matrix. Although this vector remains the standard choice in most applications, its mathematical properties have been the subject of continuous investigation \citep{crawford1985note,barzilai1997deriving}. 
Alternative priority derivation methods—including the logarithmic least squares method, optimization-based approaches and axiomatic procedures—have also been proposed, each emphasizing different theoretical properties of the resulting priority vectors \citep{ishizaka2006derive,fichtner1986,genest1994,brunelli2018survey}.

Recent years have witnessed renewed interest in the mathematical theory of pairwise comparison matrices. Significant contributions by Bozóki and collaborators have advanced the understanding of inconsistency analysis, incomplete pairwise comparison matrices and efficient weight vectors, while Csató has developed an elegant axiomatic framework for weighting methods and established several fundamental properties of the eigenvector method \citep{temesi2024incomplete,csato2024coincidence,csato2021monotonicity,csato2024right}. Complementary surveys by Ishizaka, Brunelli and their collaborators have synthesized these developments and identified several methodological and theoretical challenges that remain open \citep{ishizaka2011review,brunelli2014introduction}. Together, these works demonstrate that priority derivation from reciprocal matrices continues to be an active area of research.

Beyond comparing different weighting methods lies a more fundamental mathematical question. Rather than asking how a priority vector should be computed, one may ask whether the resulting vector is optimal with respect to the information contained in the reciprocal matrix. This viewpoint naturally leads to the notion of Pareto efficiency of priority vectors, introduced by \citet{blanquero2006inferring}. Their work established a new perspective on priority derivation by interpreting efficient vectors as nondominated approximations of the original pairwise judgements.

 \citet{blanquero2006inferring} established a graph-theoretical characterization of efficient vectors. They proved that the efficiency of a priority vector is equivalent to the strong connectivity of a digraph naturally associated with the reciprocal matrix and the vector itself. This result revealed a deep connection between Pareto optimality and graph theory, opening new directions for the structural analysis of reciprocal matrices.

The efficiency of the normalized positive eigenvector associated with the largest eigenvalue has subsequently become one of the central topics in the mathematical theory of AHP. Many researchers have studied the corresponding eigenvector by setting its first component equal
to one, the \emph{Perron vector}, because the efficiency of a priority vector is the same as all positive multiple vectors of it. Although this vector is efficient for many important classes of reciprocal matrices, it may fail to be Pareto optimal even for matrices exhibiting relatively mild inconsistency. This phenomenon has motivated a sequence of investigations devoted to identifying classes of reciprocal matrices for which this vector is guaranteed to be efficient. In particular,  \citet{abele2016efficiency} analysed matrices obtained by perturbing a single entry and its reciprocal entry of a consistent reciprocal matrix, \citet{abele2018efficiency} extended these results to double perturbations, while  \citet{fernandes2022efficiency} and  \citet{fernandes2026} considered perturbations involving three and four entries.

Graph theory has played an increasingly important role in this line of research.  \citet{fernandes2024positive} proved that the strong connectivity of the associated digraph is equivalent to the existence of a directed Hamiltonian cycle, thereby providing an alternative characterization of efficiency. Later,  \citet{fernandes2024triple} exploited this characterization to establish structural properties of inefficient Perron vectors through the existence of sink vertices in the associated digraph.

When the Perron vector is inefficient, attention naturally shifts to the entire set of efficient vectors. The systematic investigation of this set revealed that different efficient vectors may induce different rankings of the alternatives, as detailed in \citet{furtado2024efficient}. Consequently, selecting an appropriate efficient representative becomes a nontrivial optimization problem. In subsequent work \citet{furtado2024efficiency} proposed embedding the original reciprocal matrix into a larger reciprocal matrix possessing an efficient Perron vector, while \cite{fernandes2026} further analysed this construction. More recently, \citet{szadoczki2024geometric} investigated the geometry of efficient vectors in dimensions three and four, revealing additional structural properties of the efficient region.

Despite these significant advances, several aspects of the graph-theoretical structure underlying efficient and inefficient priority vectors remain poorly understood. In particular, the condensation graph of the digraph associated with a reciprocal matrix and its Perron vector does not receive attention, despite its potential to reveal the global organization of strongly connected components and to provide additional information on the structure of efficient vectors.

The main contributions of this paper are as follows.

First, we investigate the condensation graph associated with the digraph of a reciprocal matrix and its Perron vector, introducing this graph as a new object in the study of efficient priority vectors.

Second, we establish new structural properties of these condensation graphs, extending existing graph-theoretical characterizations of efficient and inefficient Perron vectors.

Third, we derive new theoretical results that clarify the relationship between strongly connected components, condensation graphs and Pareto efficiency, thereby providing a deeper understanding of the mathematical structure of reciprocal matrices.

These contributions strengthen the graph-theoretical approach to the analysis of reciprocal matrices and further illuminate the role of efficiency in priority derivation.

The remainder of the paper is organized as follows. Section 2 presents formal definitions and known results relevant to this work. It also includes remarks on the condensation graph of the digraph associated with a reciprocal matrix and a positive vector. Section 3 studies the structure of this condensation digraph.
In Section 4, we construct a reciprocal matrix from a given one with an inefficient vector by modifying a single pair of reciprocal entries, in such a way that the vector becomes efficient in the new matrix. Section 5 focuses on embedding a reciprocal matrix and an inefficient vector into a larger reciprocal matrix and a corresponding efficient vector using condensation graphs. Finally, Section 6 describes a procedure for obtaining an efficient vector from an inefficient one.

\section{Preliminaries}

\hspace{3ex}
A matrix $A=[a_{ij}]\in M_n$, where $M_n$ denotes the set of all $n\times n$ real matrices, is called a \emph{reciprocal matrix} (or \emph{pairwise comparison matrix}) if all its entries are positive and satisfy
\[
a_{ij}=\frac{1}{a_{ji}}, \qquad i,j=1,\ldots,n.
\]
In particular, $a_{ii}=1$ for every $i=1,\ldots,n$. The entry $a_{ij}$ represents the relative dominance or preference of alternative (or criterion) $i$ over alternative (or criterion) $j$ with respect to a fixed criterion (or decision goal). We denote by ${\cal PC}_n$ the set of all reciprocal matrices of order $n$.

A matrix $A\in{\cal PC}_n$ is said to be \emph{consistent} (or \emph{transitive}) if
\[
a_{ij}a_{jk}=a_{ik},\qquad i,j,k=1,\ldots,n.
\]
It is well known that every consistent matrix admits the factorization
\[
A=ww^{-1},
\]
where $w=\left[w_1,\ldots,w_n\right]^T$ is a positive vector and
\begin{equation}
w^{-1}=\left[w_1^{-1},\ldots,w_n^{-1}\right].
\label{w-1}
\end{equation}
Following \cite{saaty1977scaling}, the priority vector associated with a reciprocal matrix is usually taken as the normalized positive eigenvector corresponding to its largest eigenvalue. Every matrix in ${\cal PC}_2$ is consistent, whereas for $n>2$ consistency is not guaranteed.

When $A\in{\cal PC}_n$ is inconsistent, \cite{blanquero2006inferring} introduced the notion of \emph{efficiency} for priority vectors. A positive vector
\[
w=\left[w_1,\ldots,w_n\right]^T
\]
is said to be \emph{efficient} for $A=[a_{ij}]\in{\cal PC}_n$ if there exists no positive vector
\[
w'=\left[w'_1,\ldots,w'_n\right]^T
\]
such that
\[
\left|a_{ij}-\frac{w'_i}{w'_j}\right|
\le
\left|a_{ij}-\frac{w_i}{w_j}\right|,
\qquad 1\le i,j\le n,
\]
with strict inequality holding for at least one pair $(i,j)$. Otherwise, $w$ is called \emph{inefficient}.

An important connection between efficiency and graph theory was established by \cite{blanquero2006inferring}. They proved that a vector $w$ is efficient for $A\in{\cal PC}_n$ if and only if the associated digraph
\[
G_{A,w}=(V,E),
\]
where $V=\{1,\ldots,n\}$ and
\[
E=
\left\{
(i,j):
\frac{w_i}{w_j}\ge a_{ij},
\ i\neq j
\right\},
\]
is strongly connected.

Observe that whenever
\[
\frac{w_i}{w_j}=a_{ij},
\]
both arcs $(i,j)$ and $(j,i)$ belong to $G_{A,w}$. Moreover, for every pair of distinct vertices $i$ and $j$, at least one of the arcs $(i,j)$ or $(j,i)$ belongs to $G_{A,w}$. Hence, $G_{A,w}$ is a semicomplete digraph.

The \emph{condensation graph} of $G_{A,w}$, denoted by $CG_{A,w}$, is the digraph whose vertices are the strongly connected components of $G_{A,w}$, say
\[
S_1,S_2,\ldots,S_p.
\]
There is an arc from $S_i$ to $S_j$ ($i\neq j$) whenever there exist vertices $\ell$ of $S_i$ and $k$ of $S_j$ such that $(\ell,k)\in E$. 

Since $G_{A,w}$ is a semicomplete digraph, if  the arc  $(S_i,S_j)$ ($i\neq j$) belongs to $CG_{A,w}$, then for all vertices $\ell$ of $S_i$ and $k$ of $S_j$ the arc $(\ell,k)$ belongs to $G_{A,w}$. Moreover, the  condensation graph $CG_{A,w}$ is a transitive tournament. Consequently, $CG_{A,w}$ has a unique source vertex and a unique sink vertex.

Finally, by the Perron--Frobenius theorem for positive matrices \citep{johnson1985matrix}, the largest eigenvalue $r$ of a matrix $A\in{\cal PC}_n$, called the \emph{Perron eigenvalue}, is simple and positive, and every eigenvector associated with $r$ has nonzero entries of the same sign. Moreover,
\[
r\geq n.
\]

\section{On the Structure of $CG_{A,w}$}

\hspace{3ex}Let $w$ be the Perron vector of a reciprocal matrix $A$. If $w$ is efficient then the condensation digraph $CG_{A,w}$ consists of a single vertex. Otherwise, the  digraph $CG_{A,w}$ has more than one vertex. Understanding the structure of its source vertex provides insight into the fact that the associated digraph $G_{A,w}$
 has no source vertex, a result established in \cite{fernandes2026}. The next theorem explains this phenomenon.

\begin{theorem}\label{p1}  Let $A\in {\cal PC}_n$, with $n\geq 3$. Let $w$ be the Perron vector of $A$ and let $G_{A,w}$ be the associated digraph with $A$ and $w$. If $w$ is inefficient for $A$ then the source vertex of $CG_{A,w}$ is a strongly connected component of $G_{A,w}$ with at least three vertices.
\end{theorem}

\begin{proof}  Let $w=\left[
\begin{array}
[c]{ccc}%
w_{1} & \cdots & w_{n}%
\end{array}
\right]  ^{T}$ and let $A=[a_{ij}]$. Let $r$ be the Perron eigenvalue of $A$. Let $S_1$ be the source vertex of $CG_{A,w}$ and let $\ell$ be a vertex of $S_1$. Since $S_1$ is a source vertex of $CG_{A,w}$, for every vertex
$i\not\in S_1$, $\frac{w_\ell}{w_i}> a_{\ell i}.$ This implies that $w_\ell > a_{\ell i} w_i.$ From the $\ell$th row of $Aw=rw$ we obtain $$\sum_{j =1}^n a_{\ell j}w_j=rw_\ell.$$

If $S_1$ has only the vertex $\ell$ then $$0=\sum_{j =1,\ j\neq \ell}^n a_{\ell j}w_j+(1-r)w_\ell<(n-1)w_\ell+(1-r)w_\ell=(n-r)w_\ell.$$ Therefore, $0< (n-r)w_\ell$ which  is impossible because $r\geq n$ and $w_\ell >0$.

If $S_1$ has only two vertices $\ell$ and $h$ then $(\ell ,h)\in E$ and $(h,\ell )\in E$. This implies that $\frac{w_\ell}{w_h}\geq a_{\ell h}$ and $\frac{w_h}{w_\ell}\geq a_{h\ell }$. Since $a_{h\ell }=\frac{1}{a_{\ell h}}$, we get $w_\ell =a_{\ell h} w_h$. Therefore,
$$0=\sum_{j =1,\ j\neq \ell,\ j\neq h}^n a_{\ell j}w_j+(1-r)w_\ell+a_{\ell h}w_h<(n-2)w_\ell+(1-r)w_\ell+w_\ell=(n-r)w_\ell.$$ Again, $0< (n-r)w_\ell$ which  is impossible because $r\geq n$ and $w_\ell >0$.

Consequently, $S_1$ has at least three vertices.
\end{proof}

\vspace{1ex}

Theorem \ref{p1} guarantees that only the source  vertex of $CG_{A,w}$ must contain at least three vertices. The following example shows that no similar restriction holds for the remaining strongly connected components of $G_{A,w}$.
The matrix $$A_1=\left[\begin{array}{ccccc}1&1.1&0.7&4.5&0.5\\ \frac{1}{1.1}&1&1.5&0.8&1.1\\[1ex] \frac{1}{0.7}&\frac{1}{1.5}&1&1.1&1.7\\[1ex] \frac{1}{4.5}&\frac{1}{0.8}&\frac{1}{1.2}&1&5\\[1ex] \frac{1}{0.5}&\frac{1}{1.1}&\frac{1}{1.7}&\frac{1}{5}&1\end{array}\right]\in {\cal PC}_{5}$$ has the Perron vector $w_1=\left[\begin{array}{ccccc} 1&0.624&0.726&0.856&0.593\end{array}\right]^T$. In this case, the source vertex of $CG_{A_1,w_1}$
 corresponds to the strongly connected component of $G_{A_1,w_1}$ containing vertices $1,4$, and $5$, the sink vertex corresponds to the singleton strongly connected component $\{2\}$, and vertex $3$ forms another singleton strongly connected component.
 \vspace{1ex}

Theorem \ref{p1} has several immediate consequences concerning the structure of $G_{A,w}$.
As an immediate consequence, we recover a result of \cite{fernandes2026}.

\begin{corollary} Let $A\in {\cal PC}_n$, with $n\geq 3$ and let $w$ be its inefficient Perron vector. Then the digraph $G_{A,w}$ has no source vertex.
\end{corollary}

Another important consequence of the previous theorem is the fact that every matrix in ${\cal PC}_n$, with $n\leq 3$ has efficient Perron vector.

\begin{corollary} Let $A\in {\cal PC}_3$ and let $w$ be its Perron vector. Then $w$ is efficient for $A$.
\end{corollary}

\begin{corollary} Let $A\in {\cal PC}_4$ and let $w$ be its Perron vector. If $w$ is inefficient for $A$ then the digraph $G_{A,w}$ has a sink vertex.
\end{corollary}

 The assumption that $w$ is the Perron vector is essential. This is illustrated by the example in
\cite{fernandes2026} with the matrix $A=\left[\begin{array}{ccc}1&1&2\\ 1&1&1\\ \frac{1}{2}&1&1\end{array}\right]\in {\cal PC}_{3}$ and the vector $w=\left[\begin{array}{ccc} 1&2&3\end{array}\right]^T$. In this case, the vertex $3$ is the source vertex of $G_{A,w}$.

\section{The effect of a single perturbation on a reciprocal matrix}

\hspace{3ex}This section proves that every reciprocal matrix for which a given vector is inefficient can be transformed into a reciprocal matrix for which the same vector is efficient by modifying only a single pair of reciprocal entries.

\begin{theorem} \label{t2} Let $A\in {\cal PC}_n$, with $n\geq 4$. If the positive vector $w$ is inefficient for $A$, then there exists $A'\in {\cal PC}_n$ obtained from $A$ by modifying a single entry and its reciprocal, such that $w$ is efficient for $A'$.
\end{theorem}

\begin{proof}  Let $w=\left[
\begin{array}
[c]{ccc}%
w_{1} & \cdots & w_{n}%
\end{array}
\right]  ^{T}$ and $A=[a_{ij}]$. Let $S_1$ and $S_p$ denote, respectively, the source and sink vertices of $CG_{A,w}$. Choose vertices $s_1\in S_1$ and $s_p\in S_p$.  Recall that the arc $(s_1,s_p)$ belongs to $G_{A,w}$ whereas $(s_p,s_1)$ does not.

Let $A'=[a'_{ij}]$ be obtained from $A$ by setting  $$a'_{s_1s_p}=\frac{w_{s_1}}{w_{s_p}}\ ,\ a'_{s_ps_1}=\frac{w_{s_p}}{w_{s_1}}$$ and leaving all the remaining entries unchanged.

The digraph $G_{A',w}$ contains all the arcs of $G_{A,w}$  together with the new arc $(s_p,s_1)$. The new arc creates a directed cycle that passes through every vertex of $CG_{A,w}$.
 Hence the condensation graph collapses into a single strongly connected component, implying that $G_{A',w}$ is strongly connected and $w$ is efficient for $A'$.
\end{proof}

As an immediate consequence of Theorem \ref{t2}, if $w$ is the inefficient Perron vector of $A$, then there exists a reciprocal matrix obtained from $A$ by modifying only one pair of reciprocal entries for which the Perron vector is efficient.

\begin{example}\label{xx1} Consider the following reciprocal matrix $$A=\left[\begin{array}{cccc}1&1.2&4&9\\ &&&\\ \frac{1}{1.2}&1&7&5\\ &&&\\ \frac{1}{4}&\frac{1}{7}&1&4\\ &&&\\ \frac{1}{9}&\frac{1}{5}&\frac{1}{4}&1\end{array}\right].$$ The Perron vector of $A$ is $w=\left[
\begin{array}
[c]{cccc}%
1 & 0.99&0.26 & 0.11%
\end{array}
\right]  ^{T}$, which is inefficient since the associated digraph $G_{A,w}$ is not strongly connected

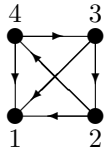
\vspace{-8ex}
 \begin{picture}(400,150)(-130,0)

\put(10,60){\circle*{6}}
\put(40,60){\circle*{6}}
 \put(10,90){\circle*{6}}
 \put(40,90){\circle*{6}}

\put(10,55){\makebox(0,0)[t]{$1$}}
\put(40,55){\makebox(0,0)[t]{$2$}}
\put(40,102){\makebox(0,0)[t]{$3$}}
\put(10,102){\makebox(0,0)[t]{$4$}}

  \put(10,60){\line(1,0){30}}
  \put(40,60){\vector(-1,0){18}}
  \put(10,60){\line(0,1){30}}
   \put(10,90){\vector(0,-1){18}}
   \put(10,60){\line(1,1){30}}
   \put(40,90){\vector(-1,-1){24}}
    \put(40,60){\line(-1,1){30}}
    \put(40,60){\vector(-1,1){24}}
     \put(40,60){\line(0,1){30}}
      \put(40,90){\vector(0,-1){18}}
      \put(10,90){\vector(1,0){18}}
    \put(10,90){\line(1,0){30}}

\end{picture}
\vspace{-13ex}

Applying Theorem \ref{t2}, we obtain that each of the following reciprocal matrices has $w$ as an efficient vector:
$$A_1=\left[\begin{array}{cccc}1&{\bf 1.01}&4&9\\ &&&\\ {\bf \frac{1}{1.01}}&1&7&5\\ &&&\\ \frac{1}{4}&\frac{1}{7}&1&4\\ &&&\\ \frac{1}{9}&\frac{1}{5}&\frac{1}{4}&1\end{array}\right], \hspace{2ex} A_2=\left[\begin{array}{cccc}1&1.2&{\bf 3.84}&9\\ &&&\\ \frac{1}{1.2}&1&7&5\\ &&&\\ {\bf \frac{1}{3.84}}&\frac{1}{7}&1&4\\ &&&\\ \frac{1}{9}&\frac{1}{5}&\frac{1}{4}&1\end{array}\right],$$
\vspace{1ex}

$$A_3=\left[\begin{array}{cccc}1&1.2&4&{\bf 8.72}\\ &&&\\ \frac{1}{1.2}&1&7&5\\ &&&\\ \frac{1}{4}&\frac{1}{7}&1&4\\ &&&\\ {\bf \frac{1}{8.72}}&\frac{1}{5}&\frac{1}{4}&1\end{array}\right].$$
\end{example}

When $w$ is the Perron vector of $A$ and is inefficient, Theorem \ref{p1} implies that the source vertex of $CG_{A,w}$ contains at least three vertices of $G_{A,w}$.
Consequently, as illustrated in the previous example, there are at least three reciprocal matrices that can be obtained from $A$.
Among all matrices obtained through Theorem \ref{t2}, we seek one that is closest to the consistent matrix $ww^T$.

To this end, consider the Frobenius norm of the matrix $(A'-ww^T)=[c_{ij}]$, namely, $$||A'-ww^T||_F=\left(\sum_{i,j=1,\ldots ,n}(c_{ij})^2\right)^{\frac{1}{2}}$$

The following theorem identifies the perturbation that minimizes the Frobenius distance to the consistent matrix $ww^{T}$.

\begin{theorem} \label{t33} Let $A\in {\cal PC}_n$, with $n\geq 4$, and let $w=\left[
\begin{array}
[c]{ccc}%
w_{1} & \cdots & w_{n}%
\end{array}
\right]  ^{T}$ be an inefficient vector for $A$. Let $S_1$ and $S_p$ denote, respectively, the source  and  sink vertices of $CG_{A,w}$.
Suppose that $i\in S_1$ and $j\in S_p$ satisfy $$\left(a_{ij}-\frac{w_i}{w_j}\right)^2+\left(a_{ji}-\frac{w_j}{w_i}\right)^2=\max_{ h\in S_1, \ l\in S_p}\left\{\left(a_{hl}-\frac{w_h}{w_l}\right)^2+\left(a_{lh}-\frac{w_l}{w_h}\right)^2\right\}.$$
Let $A'=[a'_{tr}]\in {\cal PC}_n$ be obtained from $A$ by setting  $$a'_{ij}=\frac{w_{i}}{w_{j}}\ ,\ a'_{ji}=\frac{w_{j}}{w_{i}}$$ and leaving all the remaining entries unchanged. Then $A'$ minimizes the Frobenius distance to $ww^{T}$ among all matrices obtained from $A$ as in Theorem \ref{t2}; that is,  $$||A'-ww^T||_F=\min\{||B-ww^T||_F:\ B\mbox{ is obtained from $A$ as in Theorem \ref{t2}}\}.$$
\end{theorem}

\begin{proof} Suppose that $B=[b_{tr}]$ is a matrix obtained from $A$ by setting  $$b_{hl}=\frac{w_{h}}{w_{l}}\ ,\ b_{lh}=\frac{w_{l}}{w_{h}},$$ with $h\in S_1$ and $l\in S_p$ and leaving all the remaining entries unchanged.

 Since $A'$ and $B$ differ from $A$ in exactly one reciprocal pair of entries, all the terms in the Frobenius norm are identical except those corresponding to the modified entries. Hence, $$\left(a_{ij}-\frac{w_i}{w_j}\right)^2+\left(a_{ji}-\frac{w_j}{w_i}\right)^2\geq \left(a_{hl}-\frac{w_h}{w_l}\right)^2+\left(a_{lh}-\frac{w_l}{w_h}\right)^2,$$ we conclude that $$||A'-ww^T||_F^2\leq ||B-ww^T||_F^2.$$ Therefore, the result follows.
\end{proof}

\begin{example} Using Example \ref{xx1}, we have $$\left(a_{12}-\frac{w_1}{w_2}\right)^2+\left(a_{21}-\frac{w_2}{w_1}\right)^2=\left(1.2-\frac{1}{0.99}\right)^2+\left(\frac{1}{1.2}-0.99\right)^2=0.061,$$
$$\left(a_{13}-\frac{w_1}{w_3}\right)^2+\left(a_{31}-\frac{w_3}{w_1}\right)^2=\left(4-\frac{1}{0.26}\right)^2+\left(\frac{1}{4}-0.26\right)^2=0.024,$$
$$\left(a_{14}-\frac{w_1}{w_4}\right)^2+\left(a_{41}-\frac{w_4}{w_1}\right)^2=\left(9-\frac{1}{0.11}\right)^2+\left(\frac{1}{9}-0.11\right)^2=0.008.$$ 
Therefore, by Theorem \ref{t33}, the matrix $A_1$ is the reciprocal matrix obtained from $A$ by modifying a single pair of reciprocal entries
that minimizes the Frobenius distance to $ww^T$.
    \end{example}

\section{The condensation digraph and extensions}

\hspace{3ex}Let $C\in {\cal PC}_{n+1}$ and let $i\in\{1,\ldots,n+1\}$. We denote by $C(i)$ the matrix obtained from $C$ by deleting its $i$th row and $i$th column. An {\it extension} of a matrix $A\in {\cal PC}_n$ is a matrix $B\in {\cal PC}_{n+1}$ such that $B(n+1)=A$.

Let $v=\left[
\begin{array}
[c]{cccc}%
v_{1} & \cdots & v_n&v_{n+1}%
\end{array}
\right]  ^{T}$ be a positive vector and let $i\in\{1,\ldots,n+1\}$. We denote by $v(i)$ the vector obtained from $v$ by deleting its $i$th row.  An {\it extension} of a positive vector $w$ with $n$ components is a positive vector $w'$ with $n+1$ components such that $w'(n+1)=w$.

The construction of extensions preserving or improving efficiency properties has recently attracted attention. \citet{fernandes2024triple} considered a particular extension of a triple perturbed reciprocal matrix obtained by appending a row and a column whose entries are all equal to one. They provided an example in which the Perron vector of the original matrix is inefficient, whereas the Perron vector of its extension is efficient. Moreover, their example shows that the pairwise order of the corresponding components of the two Perron vectors may or may not be preserved. 

On the other hand, \citet[Theorem~33]{furtado2024efficiency} proved that every reciprocal matrix with an inefficient Perron vector admits an extension whose Perron vector is efficient. However, their result does not address whether the first $n$ components of the efficient vector of the extension can coincide with the Perron vector of the original matrix. 

In this section, we use the condensation digraph to obtain such extensions. More precisely, starting from a reciprocal matrix with an inefficient Perron vector, we construct an extension of the reciprocal matrix and an extension of the Perron vector, efficient for this matrix.

\begin{theorem} \label{t3} Let $A\in{\cal PC}_n$, with $n\ge4$, and let $w$ be its Perron vector. Then there exists an extension $B$ of $A$ together with an extension   $w'$ of $w$ efficient for $B$.  

Moreover, there exists an extension $w'=\left[
\begin{array}
[c]{cccc}%
w'_{1} & \cdots & w'_{n}&w'_{n+1}%
\end{array}
\right]  ^{T}$ of $w$ such that $$0<w'_{n+1}<\min\{w'_i:\ 1\leq i\leq n\}.$$
\end{theorem}

\begin{proof}  Let $w=\left[
\begin{array}
[c]{ccc}%
w_{1} & \cdots & w_{n}%
\end{array}
\right]  ^{T}$ and $A=[a_{ij}]$. 

If $w$ is efficient, then $G_{A,w}$ is strongly connected. Choose two vertices $s_1$ and  $s_p$ of $G_{A,w}$.

If $w$ is inefficient, then $G_{A,w}$ is not strongly connected. Let $CG_{A,w}$ be its condensation digraph, and let $S_1$ and $S_p$ denote, respectively, the source and the sink vertices of $CG_{A,w}$. Choose vertices $s_1\in S_1$ and $s_p\in S_p$.

Let $w_{n+1}$ be any positive real number satisfying \[ w_{n+1}<\min\{w_i:1\le i\le n\}, \] and define the reciprocal matrix $B=[b_{ij}]\in{\cal PC}_{n+1}$ by \[ b_{ij}= \begin{cases} a_{ij}, & 1\le i,j\le n,\\[1mm] \dfrac{w_{s_1}}{w_{n+1}}, & i=s_1,\ j=n+1,\\[3mm] \dfrac{w_{s_p}}{w_{n+1}}, & i=s_p,\ j=n+1,\\[3mm] \dfrac{w_{n+1}}{w_{s_1}}, & i=n+1,\ j=s_1,\\[3mm] \dfrac{w_{n+1}}{w_{s_p}}, & i=n+1,\ j=s_p,\\[3mm] 1, & \text{otherwise}. \end{cases} \] Clearly, $B$ is an extension of $A$. Consider now the vector $$w'=\left[
\begin{array}
[c]{cccc}%
w_{1} & \cdots & w_{n}&w_{n+1}%
\end{array}
\right]  ^{T}.$$ Clearly, $w'$ is an extension of $w$. By construction, \[ \frac{w_{s_1}}{w_{n+1}}=b_{s_1,n+1} \quad\text{and}\quad \frac{w_{s_p}}{w_{n+1}}=b_{s_p,n+1}, \] so the digraph $G_{B,w'}$ contains the arcs $$(n+1, s_1) \quad\text{and}\quad (s_p, n+1). $$

If $w$ is efficient, since $G_{A,w}$ is strongly connected, then there exists a directed path
 from $s_1$ to every vertex, and  from every vertex to $s_p$. 

If $w$ is inefficient, since $S_1$ is a source of the condensation digraph, every vertex of $G_{A,w}$ is reachable from $s_1\in S_1$. Likewise, because $S_p$ is a sink, every vertex of $G_{A,w}$ reaches $s_p\in S_p$. 

Consequently, in both cases, every vertex of $G_{B,w'}$ is reachable from the new vertex $n+1$, and every vertex can reach $n+1$.
Hence $G_{B,w'}$ is strongly connected. Therefore, the vector $w'$ is efficient for $B$.
\end{proof}

\vspace{1ex}

The construction described in Theorem \ref{t3} can be modified so that the additional component of $w$ is assigned an arbitrary positive value.

\section{Efficient vector obtained from an inefficient}

\hspace{3ex}The main result of this section establishes a procedure to transform an inefficient vector into an efficient one for a given reciprocal matrix. Once again, this procedure uses the condensation digraph.

\begin{theorem}\label{t4} Let $A=[a_{ij}]\in {\cal PC}_n$, with $n\geq 4$. Let $w=\left[
\begin{array}
[c]{ccc}%
w_{1} & \cdots & w_{n}%
\end{array}
\right]  ^{T}$  be an inefficient vector for $A$. Let $S_1$ and $S_p$ be, respectively, the source and sink vertices of $CG_{A,w}$. Let $s_1$ be a vertex of $S_1$ and let $s_p$ be a vertex of $S_p$. Define $w'=\left[
\begin{array}
[c]{ccc}%
w'_{1} & \cdots & w'_{n}%
\end{array}
\right]  ^{T}$ such that $$w'_i=\left\{\begin{array}{ll}a_{s_ps_1}w_{s_1}&\mbox{ if }i=s_p\\ w_i&\mbox{ otherwise }\end{array}\right..$$  Then
all vertices of $S_1$ and the vertex $s_p$ belong to the same strongly connected component of $G_{A,w'}$.

Moreover, if $w'$ is still inefficient for $A$, then $s_p$ belongs to the strongly connected component that is the source vertex of $CG_{A,w'}$.
\end{theorem}

\begin{proof} Since $w$ and $w'$ differ only in the $s_p$-th component, the digraphs $G_{A,w}$ and $G_{A,w'}$ coincide on all arcs not involving $s_p$. 
By definition of $s_1\in S_1$ and $s_p\in S_p$, we have
$$\frac{w_{s_1}}{w_{s_p}} > a_{s_1 s_p}.$$
By construction of $w'$, it follows that
$$\frac{w_{s_1}}{w'{s_p}} = \frac{w_{s_1}}{a_{s_p s_1} w_{s_1}} = \frac{1}{a_{s_p s_1}} = a_{s_1 s_p},$$
using reciprocity of $A$. Hence, both arcs $(s_1,s_p)$ and $(s_p,s_1)$ are present in $G_{A,w'}$, implying that $s_1$ and $s_p$ lie in the same strongly connected component of $G_{A,w'}$.

Since all vertices in $S_1$ are mutually reachable in $G_{A,w}$ and their mutual reachability is unaffected in $G_{A,w'}$, it follows that all vertices of $S_1$ and the vertex $s_p$ belong to a single strongly connected component of $G_{A,w'}$.

For the second claim, if $w'$ remains inefficient, then the condensation graph $CG_{A,w'}$ is nontrivial. The merging of $s_p$ with $S_1$ eliminates incoming arcs to this component, implying that $s_p$ (now part of this component) is a vertex of the source vertex of $CG_{A,w'}$.
\end{proof}

\vspace{1ex}

The following result is an immediate consequence of Theorem \ref{t4}.

\begin{corollary} Let $A=[a_{ij}]\in {\cal PC}_n$, with $n\geq 4$. Let $w$  be an inefficient vector for $A$. Then there exists an efficient vector $w'$ obtained from $w$ by applying the transformation described in Theorem \ref{t4} at most $n-1$ times.
    \end{corollary}

When $w$ is the Perron vector of $A$, additional structural properties of $CG_{A,w}$ allow for a sharper bound (see Theorem \ref{p1}).

 \begin{corollary} Let $A=[a_{ij}]\in {\cal PC}_n$, with $n\geq 4$. Let $w$  be the inefficient Perron vector for $A$. Then there exists an efficient vector $w'$ obtained from $w$ by applying the transformation described in Theorem \ref{t4} at most $n-3$ times.
    \end{corollary}  

\section*{Acknowledgements}

This work is funded by national funds through the FCT – Funda\c{c}\~{a}o para a Ci\^{e}ncia e a Tecnologia, I.P., under the scope of the projects UID/297/2025 and UID/PRR/297/2025 (Center for Mathematics and Applications - NOVA Math).

\section*{Disclose statement}

No potential conflict of interest was reported by the author.

\bibliography{obras}

\end{document}